\documentclass[12pt]{amsart}

\usepackage{amsfonts}
\usepackage{amssymb}
\usepackage[T2A]{fontenc}
\usepackage[cp1251]{inputenc}
\usepackage[russian,english]{babel}

\newtheorem{teo}{Theorem}[section]
\newtheorem{lem}[teo]{Lemma}
\newtheorem{cor}[teo]{Corollary}
\newtheorem{dfn}[teo]{Definition}
\newtheorem{prop}[teo]{Proposition}
\theoremstyle{definition}
\newtheorem{rk}[teo]{Remark}

\usepackage{color}

\def\<{\langle}
\def\>{\rangle}
\DeclareMathOperator{\card}{card}
\DeclareMathOperator{\supp}{supp}
\def\be{\beta}
\def\L{\Lambda}
\def\l{\lambda}
\def\s{\sigma}
\def\a{\alpha}
\def\g{\gamma}
\def\om{\omega}
\def\e{\varepsilon}
\def\f{{\varphi}}

\def\A{{\mathcal A}}
\def\M{{\mathcal M}}
\def\cN{{\mathcal N}}

\def\KI{{\mathcal K_{\text{I}}}}
\def\KII{{\mathcal K_{\text{II}}}}
\def\KIII{{\mathcal K_{\text{III}}}}
\def\KIV{{\mathcal K_{\text{IV}}}}

 \def\N{{\mathbb N}}
 
\begin{document}

\selectlanguage{english}

\title[Locally unital $C^*$-algebras do not admit frames]
{Locally unital $C^*$-algebras do not admit frames}

\author{D.V. Fufaev}

\thanks{The work was supported by the grant 24-11-00124 of the Russian Science Foundation.}

\address{Moscow Center for Fundamental and Applied Mathematics,
\newline
Dept. of Mech. and Math., Lomonosov Moscow State University}

\email{denis.fufaev@math.msu.ru, fufaevdv@rambler.ru}

\begin{abstract}
We study nonunital $C^*$-algebras such that for any element there exists a local unit and prove that in such algebras there are no frames. This fact was previously known only for commutative algebras. Among other results, we establish some necessary properties of frames in $C^*$-algebras (which are of independent interest in the noncommutative topology), and consider several examples of $C^*$-algebras that are new in this context.
\end{abstract}

\maketitle

\section*{Introduction}

In the theory of $C^*$-Hilbert modules, frames were first introduced by Frank and Larson in the articles \cite{FrankLarson1999} and \cite{FrankLarson2002}. Also in these articles the following question was formulated: does every module admit a frame? In Hilbert space theory, the answer is always positive, and frame theory in this context remains an active area of research (see, for example, \cite{MSBAFA} or \cite{TerPoin}); for modules, the situation turned out to be more complex. Namely, in \cite{HLi2010} an example of a Hilbert $C^*$-module that does not admit frames was constructed. This result was further developed in the works \cite{Asadi2016}, \cite{AmAs2016}, \cite{FrAs2020}. 

The question of the existence of frames in a Hilbert $C^*$-module also arose after the works on $\A$-compact operators: \cite{Troitsky2020JMAA}, \cite{TroitFuf2020}, \cite{TroitFufAFA}, where uniform structures constructed by using systems close to frames were introduced and studied for the geometric description of $\A$-compactness. In particular, in \cite{TroitFuf2020} a geometric criterion for $A$-compactness of operators was proved in the case where the range of the operator is a module in which a frame exists.

An example of a module for which the mentioned criterion does not hold was found among modules that do not have frames in the work \cite{Fuf2021faa}. Namely, a topological space and the corresponding commutative $C^*$-algebra considered as a module over itself was introduced.
Thus, within the setting of this problem, even a very simple module is full of interest, for which all pathological properties are encoded in the $C^*$-algebra.
In \cite{Fuf2022Path} the following classification of topological spaces was introduced:

$\KI$ --- $\s$-compact spaces (i.e. spaces that could be covered by a countable family of compact subsets);

$\KII$ --- non-$\s$-compact spaces that have a dense $\s$-compact subset;

$\KIII$ --- spaces 
that have no dense $\s$-compact subset, that is, the complement to any $\s$-compact subset has an interior point, but the point at infinity (in a one-point compactification) may not be interior for the complement; equivalently, there exists a $\s$-compact not precompact subset;

$\KIV$ --- spaces such that the complement to any $\s$-compact subset has an interior point, and (in a one-point compactification) the point at infinity is always interior for the complement; equivalently, every $\s$-compact subset is precompact.

With this classification, several results on non-existence of frames for commutative algebras in topological terms (i.e. in terms of properties of the corresponding topological space) were also obtained. Thus, the question of how to generalize these results to the noncommutative case, in particular, in algebraic terms, arose. In \cite{Fuf2023Th} the result that for $K\in\KII$ the algebra $C_0(K)$ has no standard frames was generalized in this way by using the notion of thick element (i.e. an element that is not a zero divisor).

In the current paper our goal is to generalize the results for spaces $K\in\KIV$ that were obtained in \cite{Fuf2021faa} (namely the fact that corresponding algebras have no frames).
For this we use the notion of the local unit for an element to determine the case when this element has ``compact support'' 
(in the "spirit" of the noncommutative geometry, where a topological space corresponds to a $C^*$-algebra, Theorem \ref{commal} and Remark \ref{commalf} actually establish a correspondence between functions with compact support and elements that have a local unit).
In a purely algebraic situation the notion of local units is also considered, see, for example, \cite{Abrams} or \cite{FazNas}.
For $C^*$-algebras this approach was used by Pedersen in \cite{PedMT66} for the purpose of constructing a minimal hereditary dense ideal.

In Section 1 some preliminaries on $C^*$-algebras, Hilbert $C^*$-modules and frames
are given.

In Section 2 we introduce local units for elements of a $C^*$-algebra and obtain some of their properties.

In Section 3 we define the class of locally unital $C^*$-algebras and study the properties of this class. In addition, we give a characterization of this class in the commutative case in terms of the corresponding topological spaces and describe some examples.

In Section 4 we obtain some technical results on frames in $C^*$-algebras: we prove that the net of finite partial sums of the series of all frame elements is bounded in norm (this result in \cite{Fuf2023Th} was obtained only for at most countable frames) and that for any state on a $C^*$-algebra there exists a frame element such that this state does not vanish on this element (in the commutative case it means that a frame must separate points, an obvious topological fact).

In Section 5 we prove that if a $C^*$-algebra is not unital, but locally unital, then as a module over itself it has no frames --- in particular, it has no standard frames, so there is no stabilization property for it, i.e. this $C^*$-algebra cannot be represented as an orthogonal direct summand in the standard module. This result is a generalization to the noncommutative case of the result that was obtained in \cite[Theorem 12]{Fuf2022Path} for topological spaces from class $\KIV$.

The author is grateful to E.V. Troitsky, V.M. Manuilov, A.I. Shtern, K.L. Kozlov, A.Ya. Helemskii and A.I. Korchagin for helpful discussions.

\section{Preliminaries}

The background on $C^*$-algebras can be found in various monographs, for example, \cite{KadRin1}, \cite{Pedersen}, \cite{BratRob}.

For any $C^*$-algebra $\A$ we denote by $\dot{\A}$ its unitalization, which we assume to be equal to $\A$ if $\A$ is unital. Also, for any $C^*$-algebra $\A$ there exists an approximate unit, i.e. an increasing net of positive elements $\{e_\l\}_{\l\in\L}$, $||e_\l||\le1$, such that
$xe_\l\xrightarrow[\l\in \L]{}x$, $e_\l x\xrightarrow[\l\in \L]{}x$ in norm for any $x\in\A$.

A state on $\A$ is a positive linear functional of norm 1.

\begin{lem}\label{uniform} (\cite[Lemma 2.1]{Fuf2023Th})
Let $\psi$ be an arbitrary state on a $C^*$-algebra $\A$ and $\{e_\l\}$ an approximate unit in $\A$. Then 
$\psi(x-e_\l x)\to 0$ uniformly on bounded sets.
Similarly for $\psi(x- xe_\l)$ and $\psi(x-e_\l x e_\l)$.
\end{lem}

\begin{dfn}
An element $g$ of a $C^*$-algebra $\A$ is called thick if for any $x\in\A$ we have that $gx=0$ implies $x=0$.
\end{dfn}

This definition is inspired by the definition of the thick submodule which is introduced in \cite{Manuilov2022Thick}: a Hilbert $C^*$-submodule $\M\subset\cN$ in a Hilbert $C^*$-module $\cN$ is thick if its orthogonal complement $\M^{\bot}$ is zero, but $\M\ne\cN$.

Note that this condition is satisfied when $g$ is a strictly positive element of a $C^*$-algebra, but this property is not necessary: if an algebra is not $\s$-unital, then it has no strictly positive element, but there may be a thick element; also we can always assume that a thick element is positive, so $xg=0$ implies $x=0$ too (see \cite{Fuf2023Th} for details).

The basic theory on
Hilbert $C^*$-modules
one can find in 
\cite{Lance},\cite{MTBook},\cite{ManuilovTroit2000JMS}. Let us recall some important notions.

\begin{dfn}
\rm
Let $\A$ be a $C^*$-algebra.
A pre-Hilbert $C^*$-module over $\A$
is 
a linear space $\M$ which is
a (right) $\A$-module (with compatible scalar multiplication: $\l(xa)=(\l x)a=x(\l a)$ for $x\in\M$, $a\in\A$, $\l\in\mathbb C$)
equipped with an $\A$-\emph{inner product}
$\<.,.\>:\M\times\M\to \A$ which is a sesquilinear form such that, for any $x,y\in\M$, $a\in\A$ :
\begin{enumerate}
\item $\<x,x\> \ge 0$ and $\<x,x\> = 0$ if and only if $x=0$;
\item $\<y,x\>=\<x,y\>^*$;
\item $\<x,y\cdot a\>=\<x,y\>a$.
\end{enumerate}

On this module a norm $\|x\|=\|\<x,x\>\|^{1/2}$ is defined, and a module is called a \emph{Hilbert $C^*$-module} if it is complete with respect to this norm.

A pre-Hilbert $C^*$-module $\M$ is said to be \emph{countably generated}
if there is a countable collection of its elements such that 
 their $\A$-linear combinations are dense in $\M$.

The \emph{Hilbert sum} of Hilbert
$C^*$-modules, defined in the obvious way, will be denoted by $\oplus$.
\end{dfn}

Any $C^*$-algebra $\A$ can be considered as a module over itself (or over $\dot{\A}$) with the inner product $\<a,b\>=a^*b$.

Let us recall the notion of a frame in a Hilbert $C^*$-module (see, e.g., \cite{FrankLarson1999}, \cite{FrankLarson2002}).
Among all frames, standard ones are also considered.

\begin{dfn}\label{dfn:fr}
Let $\cN$ be a Hilbert $C^*$-module over a unital $C^*$-algebra $\A$ and $J$ be some set. A family $\{x_j\}_{j\in J}$ of elements of $\cN$ is said to be a standard frame in $\cN$ if there exist positive constants $c_1, c_2$ such that
for every $x\in\cN$
the series 
$\sum\limits_j \<x,x_j\>\<x_j,x\>$
converges in norm in $\A$
and the following inequalities hold:

 $$c_1\<x,x\>\le \sum\limits_j \<x,x_j\>\<x_j,x\>\le c_2\<x,x\>.$$
If the series converges only in the ultraweak topology (also known as $\s$-weak, see \cite[Section 2.4.1]{BratRob}), then the frame is said to be non-standard. 
Unlike the case of a standard frame, in this case
the number of nonzero elements of the series can be uncountable, the convergence in this case is considered as the convergence of a net consisting of all finite partial sums (see remarks before \cite[1.2.19]{KadRin1} and \cite[5.1.5]{KadRin1}).
We will write just ``frame'' if it is at least non-standard.
\end{dfn}

$\cN$ is considered as a module over $\dot{\A}$ in the case when $\A$ is not unital, so frames can be defined in $\cN$ as in an $\dot{\A}$-module.

There exists a criterion for non-standard frames:

\begin{lem}\label{fr_cr}
(\cite[Proposition 3.1]{HLi2010}) 
A system $\{x_j\}_J$ is a frame in $\cN$ if and only if there exist positive constants $c_1, c_2$ such that for any $x\in\cN$ and any state $\f$ on $\A$
the following inequalities hold:

$$
c_1\f(\<x,x\>)\le
\sum\limits_j \f(\<x,x_j\>\<x_j,x\>)\le
c_2\f(\<x,x\>)
$$

\end{lem}

One of the most important structural results of the frame theory is the following:
from the results of Frank and Larson (\cite[3.5, 4.1 and 5.3]{FrankLarson2002}, see also \cite[Theorem 1.1]{HLi2010}) it follows that the Hilbert $C^*$-module
$\cN$ over the $C^*$-algebra $\A$ can be represented as an orthogonal direct summand in the standard module of some cardinality over the unitalization algebra $\bigoplus\limits_{\l\in\Lambda}\dot{\A}$ if and only if there exists a standard frame in $\cN$.
Combining this fact with the Kasparov's stabilization theorem (\cite{Kasp}, or \cite[Theorem 1.4.2]{MTBook}) one can conclude that every countably generated module has a standard frame.

\section{Local units}

\begin{dfn}
A self-adjoint element $h$ in a $C^*$-algebra $\A$ is called a local unit for an element $a\in\A$ if $ah=ha=a$. 
\end{dfn}

Note that we don't require the local unit to be idempotent, since for the theory of $C^*$-algebras this is more convenient from a technical point of view. Moreover, often a $C^*$-algebra has very few idempotents. However, instead of local units of an element we can consider the minimal projection corresponding to this element which is an idempotent and which lies in the universal enveloping von Neumann algebra $\A''$, but since $\A''$ is always unital we anyway must majorize these idempotents by some element of $\A$ which may not be idempotent.

Due to the following theorem we can always assume that for a local unit $h$ the inequalities  $0\le h \le1$ hold.

\begin{prop}
If for $a\in\A$ there exists a local unit $h$, then for $a$ there exists a local unit $h'\in\A$ such that $0\le h' \le1$. 
Moreover, for any continuous function $f:\mathbb{R}\to\mathbb{R}$ such that $f(1)=1$ and $f(0)=0$ we have that $f(h)$ is a local unit for $a$ too.
\end{prop}

\begin{proof}
Set $h'=f(h)$. 
There exists a sequence $\{f_n\}$ of polynomials, $f_n(t)=\sum\limits_{k=0}^np_{n,k}t^k$, that converges to $f$ uniformly on $\s(h)$ and that satisfies the conditions $f_n(1)=1$ and $f(0)=0$ (it suffices to take any convergent sequence $\tilde{f}_n$ and define $f_n(t)=\tilde{f}_n(t)-(1-t)\tilde{f}_n(0)-t(\tilde{f}_n(1)-1)$), i.e. $\sum\limits_{k=0}^np_{n,k}=1$ and $p_{n,0}=0$. $f_n(h)\in\A$ since $f_n(0)=0$. So, we have
$$
ah'=a\lim\limits_{n\to\infty}f_n(h)=a\lim\limits_{n\to\infty}\sum\limits_{k=0}^np_{n,k}h^k=
\lim\limits_{n\to\infty}\sum\limits_{k=0}^np_{n,k}ah^k=\lim\limits_{n\to\infty}\sum\limits_{k=0}^np_{n,k}a=a.
$$
Similarly, $h'a=a$.

If we take $f$ such that
$$
f(t)=\begin{cases}
0,\ t\le0\\
t,\ t\in[0,1]\\
1,\ t\ge1.
\end{cases}
$$
then we have that $0\le h' \le1$.

\end{proof}

\begin{rk}\label{gfuncta}
By the same argument we have that a local unit for $a$ is also a local unit for $f(a)$ for any continuous on $\s(a)$ function $f$ such that $f(0)=0$.
\end{rk}

\begin{prop}\label{majlun}
Let $\A$ be a $C^*$-algebra.
\begin{enumerate}
\item[1)] Suppose that $0\le b \le a$ and $h$ is a local unit for $a$. Then $h$ is a local unit for $b$.
\item[2)] If $h$ is a local unit for $\sum\limits_{j=1}^{n}a_j$, $a_j\ge0$, then $h$ is a local unit for every $a_j$, $j=1,\dots,n$.
\item[3)] If $h$ is a local unit for $\sum\limits_{j=1}^{\infty}\frac{1}{2^j}\frac{a_j}{1+||a_j||}$, $a_j\ge0$, then $h$ is a local unit for every $a_j$, $j\in\N$.
\item[4)] If $||a_j||\le1$, $j\in\N$ and $h$ is a local unit for $\sum\limits_{j=1}^{\infty}\frac{1}{2^j}a_j$, $a_j\ge0$, then $h$ is a local unit for every $a_j$, $j\in\N$.
\end{enumerate}
\end{prop}

\begin{proof}
Obviously, 2)-4) follow from 1), so let us prove 1).

By \cite[Proposition 1.4.5]{Pedersen}, for any $\a\in(0,1/2)$ there exists $u\in\A$ such that $\sqrt{b}=ua^{\a}$.
Note that due to Remark \ref{gfuncta} $h$ is also a local unit for $a^{\a}$. So $b=ua^{\a}ua^{\a}$ and $bh=ua^{\a}ua^{\a}h=ua^{\a}ua^{\a}=b$. By applying the conjugation, $b=a^{\a}u^*a^{\a}u^*$ and, similarly, $hb=b$.
\end{proof}

Recall (\cite[Theorem 2.2]{Fuf2023Th}) that for any state $\f$ on $\A$, there exists a positive $g\in\A$ such that $gx = 0$ implies $\f(x) = 0$ for any $x\in\A$ (by applying the conjugation it is easy to see that $xg=0$ implies $\f(x) = 0$ too). We will call such an element $g$ a support of $\f$ (we do not require that a support must be minimal).

\begin{lem}\label{fah}
Let $\A$ be a $C^*$-algebra, $\f$ a state on $\A$ and $g$ a support of $\f$. If $g$ has a local unit $h$, then for any $a\in\A$ we have $\f(a)=\f(ah)=\f(ha)=\f(hah)$. In the unitalization $\dot{\A}$ we also have $\f(a(1-h))=\f((1-h)a)=0$.
\end{lem}

\begin{proof}
Every state has a unique extension to the unitalization $\dot{\A}$ (\cite[2.3.13]{BratRob}), so we can consider $\f$ as a state on $\dot{\A}$.
For any $a\in\A$ we have
$$
\f(a)=\f(ah)+\f(a(1-h)).
$$
Note that $a(1-h)g=ag-ahg=ag-ag=0$, so $\f(a(1-h))=0$ since $g$ is a support for $\f$. So, $
\f(a)=\f(ah)$, and, similarly, $\f(a)=\f(ha)$.
\end{proof}

\begin{lem}
In a unital $C^*$-algebra $\A$ a local unit for a thick element (or a strictly positive element) coincides with the unit of the algebra.
\end{lem}

\begin{proof}
Suppose that $g\in\A$ is thick or strictly positive, i.e. $g\ge0$ and $gx=0$ implies $x=0$ for any $x\in\A$. Let $h$ be a local unit for $g$. Take an arbitrary $y\in\A$. Consider $y-yh$. We have $(y-yh)g=yg-yhg=yg-yg=0$, so $y-yh=0$, i.e. $y=yh$ since $g$ is thick or strictly positive. Similarly (by considering $y-hy$) we have that $y=hy$. So, it holds for an arbitrary $y\in \A$, hence $h$ is the unit of $\A$.
\end{proof}

\begin{cor}\label{luneth}
In a non-unital $C^*$-algebra $\A$ an element with local unit cannot be thick (or strictly positive).
\end{cor}

\begin{prop}\label{herzero}
Let $\A$ be a $C^*$-algebra and suppose that $x\in\A$ has a local unit $h\in\A$. If for some state $\f$ on $\A$ we have that $\f(h)=0$, then $\f(x)=0$.
\end{prop}

\begin{proof}
By the Cauchy-Schwarz inequality (\cite[2.3.10]{BratRob}) for states we have
$$
|\f(x)|^2=|\f(xh)|^2=|\f(x\sqrt{h}\sqrt{h})|^2\le|\f(x\sqrt{h}(x\sqrt{h})^*)\f(\sqrt{h}\sqrt{h})|=0,
$$
so $\f(x)=0$.
\end{proof}

\section{Locally unital algebras}

\begin{dfn}
We will say that a $C^*$-algebra $\A$ is locally unital if every $a\in\A$ has a local unit. 
\end{dfn}

\begin{rk}
From Proposition \ref{majlun} it follows that in such algebras there exists a common local unit for any, at most countable, set of elements.
\end{rk}

The following theorem gives a description of the class of algebras of interest to us --- locally unital but not unital --- in the commutative case in terms of the underlying topological space.

\begin{teo}\label{commal}
Let $\A$ be a commutative $C^*$-algebra, i.e. $\A\cong C_0(K)$, where $K$ is a locally compact Hausdorff space. Then the following statements are equivalent:
\begin{enumerate}
\item[1)] $C_0(K)$ is locally unital, but not unital.
\item[2)] $K\in\KIV$, i.e. any $\s$-compact subset in $K$ is precompact.
\item[3)] Any function $f\in C_0(K)$ vanishes outside some compact set which depends on $f$ (i.e. $f$ has a compact support).
\end{enumerate}
\end{teo}

\begin{proof}
Implications 2) $\Leftrightarrow$ 3) were proved in \cite[Lemma 6]{Fuf2022Path}.

1) $\Rightarrow$ 3). Let $f\in C_0(K)$ be any function. There exists a local unit $h\in C_0(K)$ for $f$, i.e. $f(t)h(t)=f(t)$ for any $t\in K$. So, if $f(t)\ne0$, then $h(t)=1$. Hence, $\{t\in K: f(t)\ne0\}\subset\{t\in K: h(t)=1\}$, but the last set is compact. So, $f$ vanishes outside this compact set.

3) $\Rightarrow$ 1) Let $f\in C_0(K)$ be a function that vanishes outside some compact $K'\subset K$. Then by \cite[Corollary 1.3]{Fuf2021faa} there exists $h\in C_0(K)$ such that $h(t)=1$ for $t\in K'$. Then $f(t)h(t)=f(t)$ for any $t\in K$, i.e. $h$ is a local unit for $f$.
\end{proof}

\begin{rk}\label{commalf}
In fact we prove that an element of $C_0(K)$ has a local unit if and only if it has a compact support as a function.
\end{rk}

One of the simplest example of such topological space is 
the space of all at most countable ordinals $[0, \omega_1)$ with the order topology. This space and its properties were discussed in detail in \cite{Fuf2021faa}.

Another commutative example can be constructed as follows. Let $\L$ be a discrete topological space, $\card(\L)=\g$, where $\g$ is some uncountable cardinal, let $\a$ be an infinite cardinal such that $\a<\g$. 
Consider the set of all bounded functions on $\L$, $C_b(\L)$ (obviously they all are continuous). Then consider the subset of all functions whose support has cardinality at most $\a$, denoted by $C_{b,\a}(\L)$. 
It is easy to see that $C_{b,\a}(\L)$ is a non-unital locally unital $C^*$-algebra. The simplest case is an algebra of bounded functions with at most countable support on an uncountable discrete space.

It is interesting to give a geometric description for the space of characters of $C_{b,\a}(\L)$, i.e. to describe a space $K_{\L,\a}$ such that  $C_{b,\a}(\L)\cong C_0(K_{\L,\a})$ (see \cite[2.1.11A]{BratRob}). Note that since $C_{b,\a}(\L)$ is a closed ideal in $C_{b}(\L)\cong C(\be\L)$ it follows that $K_{\L,\a}$ is an open subset of the Stone-\v{C}ech compactification of $\L$, $\be\L$ (which may be considered as the space of characters on $C_{b}(\L)$), and for any $f\in C_b(\L)$ and any character $p$ we have $p(f)=\widehat{f}(p)$, where $\widehat{f}$ is the extension of $f$ on $\be\L$. Also note that if $S\subset\L$, then $\be S\subset\be\L$ (see \cite[Section 1]{Walker}).

\begin{teo}
$K_{\L,\a}=
\bigcup\limits_{{
\substack{S\subset\L \\  \card(S)\le\a                     }
}}\be S\subset \be\L,
$ i.e.
$C_{b,\a}(\L)\cong
C_0(\bigcup\limits_{{
\substack{S\subset\L \\  \card(S)\le\a                     }
}}\be S)
$.
\end{teo}

\begin{proof}
We need to prove that $K_{\L,\a}$ is exactly the space of characters on $C_{b}(\L)$ (i.e. elements of $\be\L$) that are not identically zero on $C_{b,\a}(\L)$.

If $p\in\be S$ for some $S\subset \L$, $\mathrm{\card}(S)\le\a$, then take the indicator function $\chi_S$ of $S$, obviously $\chi_S\in C_{b,\a}(\L)$. Since $S$ is dense in $\be S$, $\widehat{\chi_S}$ is identically equal to 1 on $\be S$ and, therefore, $p(\chi_S)=\widehat{\chi_S}(p)=1$, so $p$ is not identically zero on $C_{b,\a}(\L)$.

If $p\notin\be S$ for any $S\subset \L$ with $\card(S)\le\a$, then for any $f\in C_{b,\a}(\L)$ we have $\card(\supp(f))\le\a$, so $p\notin\be (\supp(f))$, so $p(f)=\widehat{f}(p)=0$, i.e. $p$ is zero identically on $C_{b,\a}(\L)$.
\end{proof}

Note that even the smallest $K_{\L,\aleph_0}$ (where $\L$ has the smallest uncountable cardinality and $\aleph_0$ is countable) is not homeomorphic to $[0,\om_1)$, because the cardinality of $[0,\om_1)$ is the first uncountable cardinal, while the cardinality of $K_{\L,\aleph_0}$ is not less than the cardinality of $\be\N$, which is $2^{\mathfrak c}$
(see \cite[3.2]{Walker}). 

From these commutative examples one can easily construct noncommutative examples by considering algebras of functions on corresponding topological spaces that take values in matrix algebras (or in algebras of operators on some Hilbert space).
Another noncommutative example is described by the following theorem.

\begin{teo}
Let $H$ be a non-separable Hilbert space and $T\in B(H)$ be a bounded operator. Then the following statements are equivalent:
\begin{enumerate}
\item[1)] $H_1:=\operatorname{ker}(T)^{\bot}$ is separable.
\item[2)] $H_2:=\operatorname{ran}(T)=T(H)$ is separable.
\end{enumerate}
Therefore, every operator satisfying these conditions has a local unit and the space $\A$ of all such operators forms a non-unital locally unital $C^*$-algebra.
\end{teo}

\begin{proof}
Implication 1) $\Rightarrow$ 2) is obvious since $T(H)=T(\ker(T)\oplus\ker(T)^{\bot})= T(H_1)$.

2) $\Rightarrow$ 1). We can use the fact that $\operatorname{ker}(T)^{\bot}=\overline{\operatorname{ran}(T^*)}$ (because $\operatorname{ker}(T)=\operatorname{ran}(T^*)^{\bot}$, see \cite[Theorem 12.10]{RudFun}). But $\operatorname{ran}(T^*)\subset T^*((\ker(T^*))^{\bot})$ and $(\ker(T^*))^{\bot}=\overline{\operatorname{ran}(T)}$, the last is separable, so $\operatorname{ran}(T^*)$ is separable and, hence, $\operatorname{ker}(T)^{\bot}=\overline{\operatorname{ran}(T^*)}$ is separable.
From $\operatorname{ker}(T)^{\bot}=\overline{\operatorname{ran}(T^*)}$ we also conclude that $T^*$ satisfies these conditions too.

The fact that $\A$ is a $C^*$-algebra and that it is not unital is obvious. An example of a local unit for $T\in\A$ is given by an orthogonal projection on the subspace $\overline{H_1+H_2}$.
\end{proof}

\begin{rk}
Note that in this example we also can take a construction with arbitrary cardinals: if the cardinality of an orthonormal basis in $H$ is an uncountable cardinal $\g$, we can for any infinite cardinal $\a<\g$ consider a $C^*$-algebra $B_\a(H)$ of operators such that $\operatorname{ran}(T)=T(H)$ is a space in which the cardinality of an orthonormal basis is no greater than $\a$.
The analogue
of the previous theorem is still valid since the cardinality of an orthonormal basis of a subspace does not increase
when applying a bounded operator, taking the closure of the subspace, or taking the closure of the union of at most countable family of such subspaces and since $\overline{\operatorname{ran}(T^*)}=\operatorname{ker}(T)^{\bot}$.
\end{rk}

The following theorem is a $C^*$-algebraic interpretation of the obvious topological fact that if every $\s$-compact subset of a locally compact space $K$ is precompact, then there is no $\s$-compact subset which is dense in $K$.

\begin{teo}
Suppose that a $C^*$-algebra $\A$ is not unital, but locally unital. Then $\A$ has no thick element and no strictly positive element (as a consequence, it is not $\s$-unital).
\end{teo}

\begin{proof}
Indeed, since the algebra is not unital and every element has a local unit, by Corollary \ref{luneth} no element can be thick or strictly positive.
\end{proof}

\begin{cor}
Suppose that a $C^*$-algebra $\A$ is not unital, but locally unital. Then as a module over itself $\A$ is not countably generated and has no at most countable frame.
\end{cor}

\begin{proof}
Indeed, $C^*$-algebra $\A$ as a module over itself is countably generated iff it is $\s$-unital (\cite[Proposition 2.3]{Asadi2016}). Non-existence of countable frame follows from \cite[Theorem 2.4]{Fuf2023Th}.
\end{proof}

\begin{lem}
Let $\A$ be a $C^*$-algebra. For any sequence $\{\f_m\}_{m\in\N}$ of states there exists an element $g\in\A$ that is their common support.
\end{lem}

\begin{proof}
For any $m\in\N$ there exists a support $g_m\in\A$, $g_m\ge0$, for $\f_m$. Take $g=\sum\limits_{m\in\N}\frac{1}{2^m}\frac{g_m}{1+||g_m||}\in\A$. Suppose that $xg=0$. Then $xgx^*=0$, i.e. $\sum\limits_{m\in\N}\frac{1}{2^m}\frac{xg_mx^*}{1+||g_m||}=0$. Since it is the series of positive elements, we have that $xg_mx^*=0$ for any $m\in\N$. Hence, $x\sqrt{g_m}(x\sqrt{g_m})^*=0$, so $||x\sqrt{g_m}(x\sqrt{g_m})^*||=||x\sqrt{g_m}||^2=0$ and $x\sqrt{g_m}=0$, and $x\sqrt{g_m}\cdot\sqrt{g_m}=xg_m=0$. So, since $g_m$ is a support for $\f_m$, $\f_m(x)=0$ for any $m\in\N$, i.e. $g$ is a support for all $\f_m$, $m\in\N$.
\end{proof}

\begin{lem}\label{fahg}
Suppose that a $C^*$-algebra $\A$ is not unital, but locally unital. For any sequence $\{\f_m\}_{m\in\N}$ of states there exists an element $h\in\A$, $0\le h\le1$, such that $\f_m(a)=\f_m(ha)=\f_m(ah)$ for any $a\in\A$ and any $m\in\N$.
In addition, in the unitalization $\dot{\A}$ we also have $\f_m(a(1-h))=\f_m((1-h)a)=0$.
\end{lem}

\begin{proof}
By the previous lemma there exists $g\in\A$ that is a support for any $\f_m$, $m\in\N$. Let $h\in\A$ be a local unit for $g$. By Lemma \ref{fah} we have that $\f_m(a)=\f_m(ha)=\f_m(ah)$, i.e. the required equalities hold.
\end{proof}

\section{Additional properties of frames}

The following theorem is a generalization of \cite[Lemma 2.3]{Fuf2023Th}, where only at most countable frames were considered.

\begin{teo}
Let $\A$ be a $C^*$-algebra and $\{x_j\}_{j\in J}$ be a frame in $\A$. Then the net of finite partial sums of $\sum\limits_{J}x_jx_j^*$ (i.e. the net $\{\sum\limits_{j\in J_\theta}x_jx_j^*\}_{\theta\in\Theta}$, where $J_\theta\subset J$ is finite, $\Theta$ is a set of all finite subsets of $J$) is bounded in norm.  
\end{teo}

\begin{proof}
Suppose that these sums are not bounded. Then there is a finite set $J_\theta\subset J$ such that $||\sum\limits_{j\in J_\theta}x_jx_j^*||\ge4c_2+6$, where $c_2$ is the upper frame constant. Then by \cite[Lemma 1.2]{Troitsky2020JMAA} there exists a state $\psi$ such that $|\psi(\sum\limits_{j\in J_\theta}x_jx_j^*)|\ge2c_2+3$.
Let $\{e_\l\}_{\l\in\L}$ be an arbitrary approximate unit in $\A$.
Since $\psi(e_\l (\sum\limits_{j\in J_\theta}x_jx_j^*)e_\l)\xrightarrow[\l\in \L]{}\psi(\sum\limits_{j\in J_\theta}x_jx_j^*)$ there exists an element $e_{\l_1}$ of the approximate unit such that for any $\l>\l_1$ we have 
$|\psi(e_\l (\sum\limits_{j\in J_\theta}x_jx_j^*)e_\l)|\ge2c_2+2$. 

Note that $e_\l (\sum\limits_{j\in J_\theta}x_jx_j^*)e_\l$ is a positive element of $\A$, so
 $|\psi(e_\l (\sum\limits_{j\in J_\theta}x_jx_j^*)e_\l)|=\psi(e_\l (\sum\limits_{j\in J_\theta}x_jx_j^*)e_\l)$.
Hence, $\sum\limits_{j\in J_\theta}\psi(\<e_\l ,x_j\>\<x_j,e_\l\>)=\sum\limits_{j\in J_\theta}\psi(e_\l x_jx_j^*e_\l)\ge2c_2+2$.

But $\psi(e_\l^2)\xrightarrow[\l\in \L]{}1$ (\cite[2.3.11]{BratRob}), so there exists an element $e_{\l_2}$ such that for any $\l>\l_2$ we have $\psi(\<e_\l,e_\l\>)=\psi(e_\l^2)<2$, hence 
$(c_2+1)\psi(\<e_\l,e_\l\>)\le2c_2+2$.

Take $\l$ such that $\l>\l_1$ and $\l>\l_2$. Then we have
$$
\sum\limits_{j\in J_\theta}\psi(\<e_\l ,x_j\>\<x_j,e_\l\>)\ge2c_2+2\ge(c_2+1)\psi(\<e_\l,e_\l\>).
$$

This contradicts the frame inequality: for any state $\psi$ and any element $x$ (in particular, for $x=e_\l$), we have

$$
\sum\limits_{j\in  J}\psi(\<e_\l ,x_j\>\<x_j,e_\l\>)\le c_2\psi(\<e_\l,e_\l\>),
$$
where the sum is taken over all non-zero elements, which in fact form at most countable set (due to \cite[Theorem 2.5]{Fuf2023Th}) which includes $J_\theta$.
Hence the net of sums is bounded.
\end{proof}

The following theorem is a generalization of the obvious topological fact that a frame must separate points.

\begin{teo}\label{nonempty}
Let $\A$ be a $C^*$-algebra and $\{x_j\}_{j\in J}$ be a frame in $\A$. Then 
for any state $\psi$ on $\A$ there exists a frame element $x_j$ such that $\psi(x_jx_j^*)\ne0$.
\end{teo}

\begin{proof}
Suppose that for any $j$ we have $\psi(x_jx_j^*)=0$. 

By the previous theorem the net of finite partial sums of the series $\sum\limits_{j\in J}x_jx_j^*$ is bounded, so by Lemma  \ref{uniform} there exists an element $e_{\l_1}$ of the approximate unit $\{e_\l\}_{\l\in\L}$ such that for all $\l>\l_1$ and all $\theta\in\Theta$ we have $|\psi(\sum\limits_{j\in J_\theta}x_jx_j^*)-\psi(e_\l \sum\limits_{j\in J_\theta}x_jx_j^* e_\l)|\le\frac{c_1}{4}$, where $c_1$ is the lower frame constant. Hence $\psi(e_\l \sum\limits_{j\in J_\theta}x_jx_j^* e_\l)\le\frac{c_1}{4}$ (since $e_\l \sum\limits_{j\in J_\theta}x_jx_j^* e_\l$ is positive) for any $\theta\in\Theta$ and any $\l>\l_1$. 

On the other hand there exists an element $e_{\l_2}$ of the approximate unit such that for all $\l>\l_2$ we have $\psi(\<e_\l,e_\l\>)=\psi(e_\l^2)\ge\frac{1}{2}$. Take $\l$ such that $\l>\l_1$ and $\l>\l_2$. Then since $\{x_j\}$ is a frame we have

$$
\frac{c_1}{2}\le
c_1\psi(\<e_\l,e_\l\>)\le\sum\limits_{j=1}^\infty\psi(\<e_\l,x_j\>\<x_j,e_\l\>),
$$
where the sum is taken over all non-zero elements, which in fact form at most countable set due to the definition of a frame. 
But
$
\sum\limits_{j=1}^\infty\psi( \<e_\l, x_j\>\<x_j, e_\l\>)=
\sup\limits_{\theta\in\Theta}\psi(e_\l \sum\limits_{j\in J_\theta}x_jx_j^* e_\l)
\le\frac{c_1}{4}$.
A contradiction, so there exists $j\in J$ such that $\psi(x_jx_j^*)>0$.
\end{proof}

\section{Non-existence of frames}

\begin{teo}
Suppose that a $C^*$-algebra $\A$ is not unital, but locally unital. Then as a module over itself $\A$ has no frames.
\end{teo}

\begin{proof}
Assume that there exists a frame $\{x_j\}_{j\in J}$ in $\A$.
Take an arbitrary state $\psi_1$ on $\A$.
There is a non-empty (due to Theorem \ref{nonempty}) at most countable (due to \cite[Theorem 2.5]{Fuf2023Th}) set $\{x_j\}_{j\in J_1}$ of elements of the frame such that $\psi_1(x_jx_j^*)\ne0$.
That is, if $j\in J\setminus J_1$, then $\psi_1(x_jx_j^*)=0$.
Let $h_1$ be a local unit for all $x_jx_j^*$, $j\in J_1$.

Assume that we have already found states $\psi_1,\dots,\psi_n$, positive elements $h_1,\dots,h_n$ and index sets
$J_1,\dots,J_n\subset J$ such that 
$\psi_i(x_jx_j^*)=0$ for all $j\in\bigcup\limits_{l=1}^{i-1}J_l$ for $i=2,\dots,n$, $\psi_l(x_jx_j^*)\ne0$ only for $j\in J_l$ (as a consequence, different sets $J_l$ do not intersect), $h_l$ is a local unit for all $x_jx_j^*$, $j\in J_l$.

Since $\A$ is locally unital, there exists a local unit $H_n$ for all $h_l$, $l=1,\dots,n$. Since $H_n$ is not strictly positive, there exists a state $\psi_{n+1}$ on $\A$ such that $\psi_{n+1}(H_n)=0$, and, hence (by Proposition \ref{herzero}), $\psi_{n+1}(h_l)=0$ for $l=1,\dots,n$, and $\psi_{n+1}(x_jx_j^*)=0$ for $j\in\bigcup\limits_{l=1}^{n}J_l$.

As in the case when $n=1$, there exists a non-empty at most countable set $\{x_j\}_{j\in J_{n+1}}$ of elements of the frame such that $\psi_{n+1}(x_jx_j^*)\ne0$ (and hence $J_{n+1}$ does not intersect any $J_{l}$, $l=1,\dots,n$, since $\psi_{n+1}(x_jx_j^*)=0$ for $j\in\bigcup\limits_{l=1}^{n}J_l$).
There also exists a local unit $h_{n+1}$ for all $x_jx_j^*$, $j\in J_{n+1}$.

By induction, we can continue this construction for any $n\in\N$.

The sequence $\{\psi_n\}_{n\in\N}$ has a limit point $\psi_0$ in the space of states on the unitalization $\dot\A$ with respect to the weak$^*$ topology (see \cite[2.3.15]{BratRob}).
Let us show that $\psi_0$ in fact is a state on $\A$, i.e. its restriction to $\A$ has norm 1. 

By Lemma \ref{fahg} there exists a positive $h\in\A$ such that $\psi_n(x)=\psi_n(xh)$ for all $n\in\N$. It also holds for $\psi_0$. Indeed, for any $x\in\A$, $\e>0$ there exists $k\in\N$ such that $\psi_k\in B_{\e,x,xh}(\psi_0)$, i.e. $|\psi_0(x)-\psi_k(x)|<\e$ and $|\psi_0(xh)-\psi_k(xh)|<\e$. Then

$$
|\psi_0(x)-\psi_0(xh)|\le|\psi_0(x)-\psi_k(x)|+|\psi_k(xh)-\psi_0(xh)|<2\e.
$$

Since this holds for any $x\in\A$, $\e>0$ we have that $\psi_0(x)=\psi_0(xh)$ for any $x\in\A$.

Now take an approximate unit $\{e_\l\}_{\l\in\L}$ in $\A$. Then

$$
\psi_0(e_\l)=\psi_0(e_\l h)\xrightarrow[\l\in \L]{}\psi_0(h)
$$
since $e_\l h\xrightarrow[\l\in \L]{}h$ in norm. On the other hand, for all $n\in\N$

$$
1\xleftarrow[\l\in \L]{}\psi_n(e_\l)=\psi_n(e_\l h)\xrightarrow[\l\in \L]{}\psi_n(h),
$$
so $\psi_n(h)=1$ for all $n\in\N$ and hence $\psi_0(h)=1$ since $\psi_0$ is a limit point of 
$\{\psi_n\}$ in the weak$^*$ topology. 
So, $\psi_0(e_\l)\xrightarrow[\l\in \L]{}1$ and hence $\psi_0$ is a state on $\A$ (due to \cite[2.3.11]{BratRob}). Let us show now that $\psi_0(x_jx_j^*)=0$ for all $j\in J$ which will contradict the Theorem \ref{nonempty}.

First let it be that $j\in J\setminus\bigcup\limits_{l=1}^{\infty}J_{l}$. Then $\psi_n(x_jx_j^*)=0$ for all $n\in\N$. Hence,
$\psi_0(x_jx_j^*)$ must equal 0. Otherwise, if $\psi_0(x_jx_j^*)=q>0$ then in any neighborhood of $\psi_0$  in weak$^*$ topology of the form $B_{q/2,x_jx_j^*}(\psi_0)=\{\f:|\psi_0(x_jx_j^*)-\f(x_jx_j^*)|<q/2\}$ there exists $\psi_n$ such that $|\psi_n(x_jx_j^*)-\psi_0(x_jx_j^*)|=q$. A contradiction.

Let now $j\in\bigcup\limits_{l=1}^{\infty}J_{l}$, i.e. $j\in J_k$ for some $k\in\N$, and suppose that $\psi_0(x_jx_j^*)\ne0$. 
Hence, $\psi_{k+l}(x_jx_j^*)=0$ for all $l\in\N$ (because $\psi_{k+l}$ vanishes on $h_k$ and, hence, on $x_jx_j^*$) and $\psi_0$ is still a limit point for the sequence $\{\psi_n\}_{n=k+1}^\infty$ in weak$^*$ topology, and then $\psi_0(x_jx_j^*)=0$ as in the previous case.

Thus,
$\{x_j\}_{j\in J}$ is not a frame, which contradicts the assumption.
\end{proof}

\begin{cor}
A non-unital locally unital $C^*$-algebra $\A$ cannot be represented as an orthogonal direct summand of a standard module $\bigoplus\limits_{\l\in\Lambda}\dot{\A}$.
\end{cor}


\begin{thebibliography}{10}



\bibitem{FrankLarson1999}
{\sc M. Frank, D.~R. Larson}.
\newblock A module frame concept for {H}ilbert {$C^\ast$}-modules.
\newblock {\em The functional and harmonic analysis of wavelets and frames
  ({S}an {A}ntonio, {TX}, 1999)}, volume 247 of {\em Contemp. Math.} (1999), 207--233. Amer. Math. Soc., Providence, RI.

\bibitem{FrankLarson2002}
{\sc M. Frank, D.~R. Larson}.
\newblock Frames in {H}ilbert {$C^\ast$}-modules and {$C^\ast$}-algebras.
\newblock {\em J. Operator Theory}, {\bf 48}:2 (2002), 273--314.





\bibitem{MSBAFA}
{\sc A. Mohammed, K. Samir, N. Bounader}.
\newblock K-frames for Krein spaces.
\newblock {\em Ann. Funct. Anal.}, {\bf 14}, 10 (2023). https://doi.org/10.1007/s43034-022-00223-3

\bibitem{TerPoin}
{\sc P.A. Terekhin}.
\newblock Frames for Hilbert spaces with respect to $\ell^1$-sum of finite-dimensional spaces.
\newblock {\em Poincare Journal of Analysis $\&$ Applications}, {\bf 10}:3 (2023). https://doi.org/10.46753/pjaa.2023.v010i03.002









\bibitem{HLi2010}
{\sc H. Li}.
\newblock A {H}ilbert ${C}^*$-module admitting no frames.
\newblock {\em Bull. Lond. Math. Soc.}, {\bf 42}:3 (2010), 388--394.







\bibitem{Asadi2016}
{\sc M.B. Asadi}.
\newblock {Frames in right ideals of ${C}^*$-algebras}.
\newblock {\em Bull. Iranian Math. Soc.}, {\bf 42}:1 (2016), 61--67.

\bibitem{AmAs2016}
{\sc  M. Amini, M.B. Asadi, G. Elliott, F. Khosravi}.
\newblock Frames in Hilbert $C^*$-modules and Morita equivalent $C^*$-algebras.
\newblock {\em Glasg. Math. J.}, {\bf 59}:1 (2017), 1--10.

\bibitem{FrAs2020}
{\sc  M.B. Asadi, M. Frank, Z. Hassanpour-Yakhdani}.
\newblock Frame-Less Hilbert $C^*$-modules II.
\newblock {\em Complex Anal. Oper. Theory}, {\bf 14}:32 (2020).






\bibitem{Troitsky2020JMAA}
{\sc E.~V. Troitsky}.
\newblock Geometric essence of ``compact'' operators on {H}ilbert {C}*-modules.
\newblock {\em Journal of Mathematical Analysis and Applications}, {\bf 485}:2 (2020), 123842.

\bibitem{TroitFuf2020}
{\sc E.V. Troitsky, D.V. Fufaev}.
\newblock Compact Operators and Uniform Structures in Hilbert ${C}^*$-Modules.
\newblock {\em Funct. Anal. Its Appl.}, {\bf 54} (2020), 287--294.

\bibitem{TroitFufAFA}
{\sc E.V. Troitsky, D.V. Fufaev}.
\newblock A new uniform structure for Hilbert ${C^*}$-modules.
\newblock {\em Ann. Funct. Anal.}, {\bf 15}, 63 (2024). https://doi.org/10.1007/s43034-024-00368-3





\bibitem{Fuf2021faa}
{\sc D.~V. Fufaev}.
\newblock A {H}ilbert {C}*-module with extremal properties.
{\em Funct. Anal. Its Appl.}, {\bf 56} (2022), 72--80.

\bibitem{Fuf2022Path}
{\sc D.~V. Fufaev}.
\newblock Topological and {F}rame {P}roperties of {C}ertain {P}athological ${C}^*$-Algebras.
\newblock {\em Russ. J. Math. Phys.}, {\bf 29} (2022), 170--182.

\bibitem{Fuf2023Th}
{\sc D.~V. Fufaev}.
\newblock {Thick Elements and States in ${C^*}$-Algebras in View of Frame Theory.}
\newblock {\em Russ. J. Math. Phys.}, {\bf 30} (2023), 184--191.



\bibitem{Manuilov2022Thick}
{\sc V.~M. Manuilov }.
\newblock {On extendability of functionals on {H}ilbert {C}*-modules.}
\newblock {\em Math. Nachr.}, {\bf 297} (2024), 998--1005. https://doi.org/10.1002/mana.202200471




\bibitem{Abrams}
{\sc G.~D. Abrams}.
\newblock Morita equivalence for rings with local units.
\newblock {\em Communications in Algebra}, {\bf 11}:8 (1983), 801--837. https://doi.org/10.1080/00927878308822881

\bibitem{FazNas}
{\sc  Z. Fazelpour, A. Nasr-Isfahani}.
\newblock Morita Equivalence and Morita Duality for Rings with Local Units and the Subcategory of Projective Unitary Modules.
\newblock {\em  Appl Categor Struct}, {\bf 32}:10 (2010). https://doi.org/10.1007/s10485-024-09764-1

\bibitem{PedMT66}
{\sc G.~K. Pedersen}.
\newblock {Measure Theory for ${C^*}$ Algebras}
\newblock {\em Mathematica Scandinavica}, {\bf 19} (1966), 131--145.



\bibitem{KadRin1}
{\sc  R.V. Kadison, J.R. Ringrose}.
\newblock {\em Fundamentals of the Theory of Operator Algebras. Volume I: Elementary theory}.
\newblock Am. Math. Soc., Providence, 1997.

\bibitem{Pedersen}
{\sc  G. K. Pedersen}.
\newblock {\em ${C}^*$-algebras and their automorphism groups},
\newblock London Mathematical Society. Monographs; 14, Academic Press, 1979.

\bibitem{BratRob}
{\sc  O. Bratteli, D.W. Robinson}.
\newblock {\em Operator Algebras and Quantum Statistical Mechanics 1},
\newblock Springer-Verlag, New York Heidelberg Berlin, 1979.



\bibitem{Lance}
{\sc E.~C. Lance}.
\newblock {\em Hilbert {C*}-modules - a toolkit for operator algebraists},
  volume 210 of {\em London Mathematical Society Lecture Note Series}.
\newblock Cambridge University Press, England, 1995.

\bibitem{MTBook}
{\sc V.M. Manuilov, E.V. Troitsky}.
\newblock {\em Hilbert ${C}^{*}$-Modules}.
\newblock American Mathematical Society, Providence, R.I., 2005.

\bibitem{ManuilovTroit2000JMS}
{\sc V.~M. Manuilov, E.~V. Troitsky}.
\newblock Hilbert ${C}^*$- and ${W}^*$-modules and their morphisms.
\newblock {\em Journal of Mathematical Sciences}, {\bf 98}:2 (2000), 137--201.




\bibitem{Kasp}
{\sc G.~G. Kasparov}.
\newblock {H}ilbert {C*}-modules: theorems of {S}tinespring and {V}oiculescu.
\newblock {\em J. Operator Theory} {\bf 4}:1(1980), 133--150.


\bibitem{Walker}
{\sc  R.C. Walker}.
\newblock {\em The Stone-\v{C}ech Compactification}.
\newblock Springer-Verlag, Berlin, Heidelberg, 1974.



\bibitem{RudFun}
{\sc  W. Rudin}.
\newblock {\em Functional Analysis (2nd ed.)}.
\newblock McGraw-Hill, New York, 1991.







\end{thebibliography}
\end{document}